\documentclass[a4paper,11pt]{article}
\usepackage{lineno}
\usepackage{amsmath}
\usepackage{graphicx}
\usepackage{setspace}
\usepackage{subfigure}
\usepackage{epsfig}
\usepackage{amsfonts}
\usepackage{natbib}
\usepackage{color}
\usepackage{epstopdf}
\usepackage{xcolor}
\usepackage{colortbl}
\usepackage{multirow}
\usepackage{booktabs}
\usepackage{lscape}
\usepackage{enumerate}
\usepackage{natbib}
 \bibpunct[, ]{(}{)}{,}{a}{}{,}%
\usepackage[a4paper,top=1in,bottom=1in,left=1in,right=1in]{geometry}
\usepackage{soul}
\usepackage{authblk}

\onehalfspace

\newtheorem{definition}{Definition}[section]
\newtheorem{theorem}{Theorem}[section]

\newtheorem{proposition}{Proposition}[section]
\newtheorem{corollary}{Corollary}[section]
\newtheorem{example}{Example}[section]
\newenvironment{proof}[1][Proof:]{\begin{trivlist}
\item[\hskip \labelsep {\bfseries #1}]}{\end{trivlist}}

\newcommand{\bX}{ \mathbf{X} }
\newcommand{\bx}{ \mathbf{x} }

\newcommand{\bp}{ \mathbf{p} }
\newcommand{\bA}{ \mathbf{A} }
\newcommand{\ba}{ \mathbf{a} }

\newcommand{\bb}{ \mathbf{b} }
\newcommand{\bC}{ \mathbf{C} }
\newcommand{\bc}{ \mathbf{c} }
\newcommand{\bD}{ \mathbf{D} }

\newcommand{\beee}{ \mathbf{e} }

\newcommand{\bff}{\mathbf{f} }

\newcommand{\bz}{ \mathbf{z} }
\newcommand{\bu}{ \mathbf{u} }
\newcommand{\bl}{ \mathbf{l} }

\newcommand{\by}{ \mathbf{y} }

\newcommand{\balpha}{ \boldsymbol{\alpha} }
\newcommand{\bdelta}{ \boldsymbol{\delta} }

\newcommand{\bsigma}{ \boldsymbol{\sigma} }
\newcommand{\bepsilon}{ \boldsymbol{\epsilon} }
\newcommand{\brho}{ \boldsymbol{\rho} }

\newcommand{\bpi}{ \boldsymbol{\pi} }

\newcommand{\bgamma}{ \boldsymbol{\gamma} }

\newcommand{\bzero}{ \mathbf{0} }

\begin{document}

\title{Trade-off preservation in inverse multi-objective convex optimization}
\date{}
\author[1]{\small Timothy C. Y. Chan}
\author[2]{\small Taewoo Lee}

\affil[1]{\footnotesize Mechanical and Industrial Engineering, University of Toronto, ON M5S 3G8, Canada}
\affil[2]{\footnotesize Industrial Engineering, University of Houston, TX 77204, USA}

\maketitle 

\begin{abstract}
We present a new inverse optimization methodology for multi-objective convex optimization that accommodates an input solution that may not be Pareto optimal and determines a weight vector that produces a Pareto optimal solution that approximates the input solution and preserves the decision maker's intention encoded in it. We introduce a notion of trade-off preservation, which we use as a measure of similarity for approximating the input solution, and show its connection with minimizing an optimality gap. Our inverse model maintains the complexity of the traditional inverse convex models. We propose a linear approximation to the model and a successive linear programming algorithm that balance between trade-off preservation and computational efficiency, and show that our model encompasses many of the existing models from the literature.  We demonstrate the proposed method using clinical data from prostate cancer radiation therapy.
\end{abstract}

\section{Introduction}
Given a feasible solution to an (forward) optimization problem, the inverse optimization problem aims to determine parameter values -- typically objective function parameters -- that make the given solution optimal.  Classical inverse approaches leverage duality to derive tractable inverse problems that retain the complexity of the underlying forward problem \citep{Ahuja,Iyengar_con}. However, these inverse models return a trivial solution (e.g., a coefficient vector of all zeros for inverse linear optimization) if the given feasible solution is not a candidate to be optimal for the forward problem.

In general, there is no guarantee that a given solution is exactly optimal for the assumed forward problem.  An implemented solution may have been adjusted after optimization, it may have been derived heuristically and is only approximately optimal, or it may simply be a noisy observation of a pristine solution.  Another possibility is that the assumed forward model is itself a simplification of a complex system for which an observed solution is near optimal. When inverse problems involve such noise and uncertainty around the model and data, it is important to determine an objective function that replicates as closely as possible the intentions of the decision maker who generated the given solution. Recent studies have generalized the classical inverse models to overcome the issue of given solutions not being exactly optimal. For example, \cite{troutt}, \cite{Keshavarz}, \cite{chow}, \cite{chan_giop}, \cite{bertimas_equilibrium}, \cite{Aswani17}, and \cite{Chan2017} developed approximate inverse optimization models that impute model parameters that make the observed solutions minimally suboptimal. In this paper, we bring together the ideas of~\cite{Keshavarz} and~\cite{chan_giop} and develop a new inverse optimization model for multi-objective convex optimization where the given solutions are not Pareto optimal.

In multi-objective optimization, decision making is typically based in the objective space, i.e., the space where the vector of objective values resides. Deciding between different solutions in this space involves examining the trade-off in objective values between different points on the Pareto frontier. With a weighted objective function, as is common in convex multi-objective problems, these points are generated by solving the forward problem with different weight vectors, which explicitly quantify the trade-off in the objectives deemed acceptable by the decision maker. Conversely, without access to the weights, and only observing a solution on the Pareto frontier, it is possible using classical inverse optimization methods to reverse engineer the weight vector that generated the solution and therefore determine the decision maker's intention with respect to trade-offs. However, if the observed solution is not on the Pareto frontier, approximate inverse optimization models like the ones mentioned above can be applied. Given a non-Pareto solution, such a model will return a weight vector that generates a Pareto optimal solution. We propose that an appropriate inverse optimization model should return a weight vector and corresponding Pareto optimal solution that differs from the input (non-Pareto) solution in a way that is consistent with the trade-offs the original decision maker had in mind. Mathematically, this means the Pareto optimal solution should have adjusted objective values that are ``consistent'' across all the objectives -- a concept we refer to as \emph{trade-off preservation}. In this paper, we will provide a formal definition of trade-off preservation, use it as a measure of similarity between different solutions, and develop an inverse model that preserves the trade-off encoded in the given solution.

Our work generalizes the approach of~\cite{chan_giop} by considering convex multi-objective optimization problems. Also, our notion of trade-off preservation is general enough to represent various ways to characterize trade-offs across multiple objectives. For example, our model can be specialized to the duality gap minimization approaches in~\cite{chan_giop} and the single-objective approximate inverse convex model in~\cite{Keshavarz}. We show that existing inverse optimization models designed for single-objective optimization which could potentially be used for multi-objective optimization, e.g., \cite{Keshavarz}, may not take into consideration trade-offs across multiple objectives encoded in the given solution. As in~\cite{Keshavarz} and~\cite{chan_giop}, we assume that a set of objectives is pre-specified.  Our contributions are:
\begin{itemize}
\item[(1)]  We generalize previous inverse optimization approaches and develop a new inverse convex multi-objective optimization model that is itself a convex problem, accommodates any input solution, and determines a nonzero weight vector that preserves the trade-off encoded in the input solution.  We introduce a notion of trade-off preservation that is generally applicable to multi-objective optimization and prove that some special cases of trade-off preservation are equivalent to the concept of duality gap minimization in inverse optimization.
\item[(2)]  We propose an efficient linear approximation of the proposed inverse problem as well as a successive linear programming algorithm that bridges the exact and approximate methods.  Based on the linear approximation, we propose a general inverse convex multi-objective optimization framework that encompasses many of the inverse models from the literature.
\item[(3)]  We demonstrate the application of our inverse optimization model to a clinical treatment planning problem in prostate cancer radiation therapy. Using a clinical dataset, we show that weights that preserve the trade-off encoded by the given objective values produce treatments that maintain the clinical quality of the original treatments across all relevant metrics.  We  show that inverse models that are not trade-off-preserving may lead to treatments that deviate substantially from the original treatments and violate clinical acceptability criteria.
\end{itemize}

\section{Background}
We first define a canonical multi-objective convex optimization problem as the forward problem.  Then, we briefly review the inverse optimization models from~\citet{Iyengar_con},~\citet{Keshavarz}, and~\citet{chan_giop}.

\subsection{Forward optimization problem}
Let $f_k:\mathbb{R}^n\rightarrow\mathbb{R},k=1,\ldots,K$ and $g_l:\mathbb{R}^n\rightarrow\mathbb{R},l=1,\ldots,L$ be convex functions.  Let $\bx\in\mathbb{R}^n, \bA\in\mathbb{R}^{m\times n},$ and $\bb\in\mathbb{R}^m$.  We define the forward optimization problem (FOP) as
\begin{subequations}\label{eq:fop}
\begin{align}
\textrm{FOP}(\balpha):\quad \underset{\bx}{\text{minimize}} & \quad \sum_{k=1}^K\alpha_k f_k(\bx)\\
\text{subject to} & \quad g_l(\bx) \le 0,\quad l=1,\ldots,L,\\
& \quad \bA\bx = \bb,
\end{align}
\end{subequations}
where $\alpha_k$ is the weight for the $k$-th objective function.  Let $\bX$ be the feasible region of~\eqref{eq:fop}. We assume $\balpha\in\mathbb{R}_+^K\backslash\{\bzero\}$, $f_k(\bx) > 0, k=1,\ldots,K$ for $\bx\in\bX$,  and $\bA$ has full rank. We also assume that Slater's condition holds \citep{boyd_convex}.  We define $\Omega(\balpha)$ to be the set of optimal solutions to FOP($\balpha$) and $\Omega := \bigcup_{\balpha\in\mathbb{R}_+^K\backslash\{\bzero\}}\Omega(\balpha)$.  A solution $\bx\in\bX$ is (weakly) Pareto optimal
if there is no other $\by\in\bX$ such that $f_k(\by)<f_k(\bx),$ for all $k=1,\ldots,K$.  It is known that for a convex multi-objective optimization problem, the set $\Omega$ consists of all Pareto optimal solutions~\citep{moobook3}.  For any $S \subseteq \bX$, we write $\bff(S) = \{(f_1(\bx), \ldots, f_K(\bx)) \, | \, \bx \in S\}$.  We denote $\bff(\bX)$ as the feasible region in the objective space and the set $\bff(\Omega)$ as the Pareto set.

\subsection{Inverse conic optimization by~\citet{Iyengar_con}}\label{sec:Iyengar_con}
We begin by illustrating the approach of \citet{Iyengar_con} using our forward problem~\eqref{eq:fop}. Given $K$ prespecified objectives and a solution $\hat\bx\in\bX$, assumed to be a regular point (\cite{nlpbook}; pp.204--207), a weight vector that makes $\hat\bx$ optimal can be found by solving the following problem:
\begin{subequations}\label{eq:iop_iyengar}
\begin{align}
\underset{\balpha,\bsigma,\bpi}{\text{minimize}} & \quad 0\\
\text{subject to} & \quad \sum_{k=1}^K\alpha_k \nabla_{\!\!\bx} f_k(\hat\bx) + \sum_{l=1}^L\sigma_l \nabla_{\!\!\bx} g_l(\hat\bx) - \bA'\bpi = \bzero,\\
&\quad \sigma_l g_l(\hat\bx) =0,\quad l=1,\ldots,L,\\
&\quad \balpha\ge \bzero,\;\bsigma \ge \bzero.
\end{align}
\end{subequations}
Constraints in problem~\eqref{eq:iop_iyengar} correspond to the KKT conditions for the forward problem~\eqref{eq:fop} with Lagrange multipliers $\bsigma$ and $\bpi$.  If $\balpha^*$ is an optimal solution that arises from solving~\eqref{eq:iop_iyengar}, then $\hat\bx\in\Omega(\balpha^*)$.
\citet{Iyengar_con} used an objective function $||\balpha - \hat\balpha||$ where $\hat\balpha$ is a given weight vector. However, we omit it to allow for an objective of minimizing ``residuals'' to be introduced.  Note that an arbitrary $\hat\bx$ need not be in $\Omega$, in which case formulation~\eqref{eq:iop_iyengar} returns $\balpha^* = \bzero$ as the unique solution (see Example~\ref{ex:iyengar}).

\begin{example}\label{ex:iyengar}
Consider the following bi-objective convex optimization problem:
\begin{subequations}\label{eq:ex_fop}
\begin{align}
\underset{\bx}{\text{minimize}} & \quad \alpha_1 f_1(x_1,x_2) + \alpha_2 f_2(x_1,x_2)\\
\text{subject to} & \quad (x_1-2)^2+(x_2-2)^2 - 1 \le 0,
\end{align}
\end{subequations}
where $f_1(x_1,x_2) = 4 x_1^2 + x_2^2$ and $f_2(x_1,x_2) = x_1^2 + 4 x_2^2$.  Constraints of the corresponding inverse problem given an input solution $\hat\bx$ are:
\begin{subequations}\label{eq:ex_iop}
\begin{align}
&(4\alpha_1 + \alpha_2)\hat{x}_1 + (\hat{x}_1-2)\sigma = 0,\\
&(\alpha_1 + 4\alpha_2)\hat{x}_2 + (\hat{x}_2-2)\sigma = 0,\\
&((\hat{x}_1-2)^2+(\hat{x}_2-2)^2 - 1)\sigma = 0,\\
&\balpha\ge \bzero,\sigma \ge 0.
\end{align}
\end{subequations}
\begin{figure}\centering
\includegraphics[width=73mm]{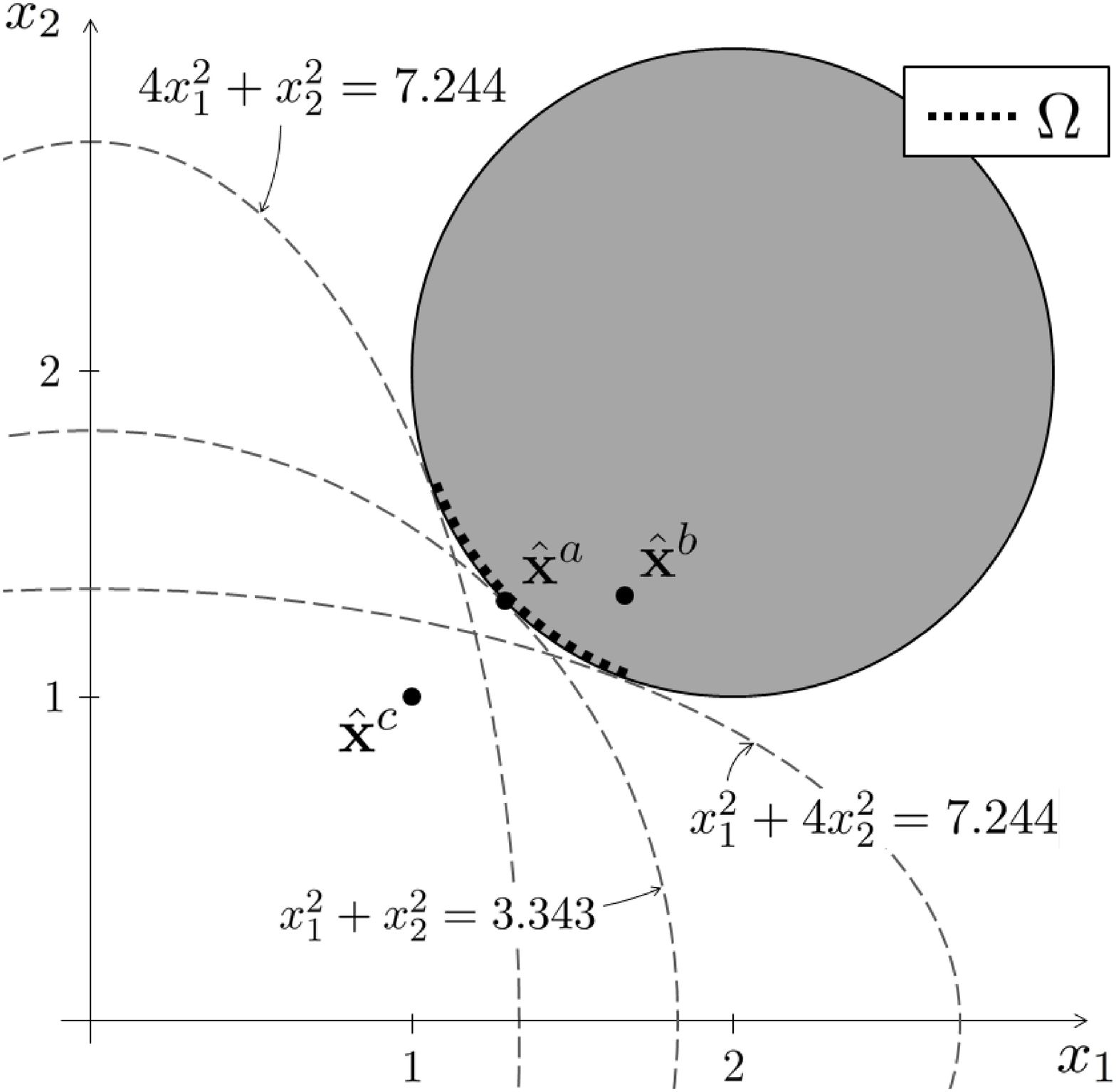}
\caption{Illustration of Example~\ref{ex:iyengar}}\label{fig:example}
\end{figure}
Consider the following three points: $\hat\bx^a = ((4-\sqrt{2})/2,(4-\sqrt{2})/2) \in \Omega$, $\hat\bx^b = (1.7,1.3) \in \bX \setminus \Omega$, and $\hat\bx^c = (1,1) \not\in \bX$ (see Figure~\ref{fig:example}). For $\hat\bx^a$, constraints~\eqref{eq:ex_iop} become:
\begin{subequations}\label{eq:ex_iop1}
\begin{align}
&(4-\sqrt{2})(4\alpha_1 + \alpha_2) - \sqrt{2}\sigma = 0,\\
&(4-\sqrt{2})(\alpha_1 + 4\alpha_2) - \sqrt{2}\sigma = 0,\\
&\balpha\ge \bzero, \sigma \ge 0,
\end{align}
\end{subequations}
which hold when $\alpha_1 = \alpha_2 \ge 0$.   For $\hat\bx^b$, constraints~\eqref{eq:ex_iop} become:
\begin{subequations}\label{eq:ex_iop2}
\begin{align}
& 1.7(4\alpha_1 + \alpha_2) - 0.3\sigma = 0,\\
& 1.3(\alpha_1 + 4\alpha_2) - 0.7\sigma = 0,\\
&\balpha\ge \bzero, \sigma = 0,
\end{align}
\end{subequations}
which are satisfied only when $\alpha_1 = \alpha_2 = 0$.  For $\hat\bx^c$,
constraints~\eqref{eq:ex_iop} become:
\begin{subequations}\label{eq:ex_iop3}
\begin{align}
& (4\alpha_1 + \alpha_2) - \sigma = 0,\\
& (\alpha_1 + 4\alpha_2) - \sigma = 0,\\
&  \balpha\ge \bzero, \sigma = 0,
\end{align}
\end{subequations}
which again are satisfied only when $\alpha_1 = \alpha_2 = 0$.  In all cases, $\balpha = \bzero$ is a feasible inverse solution.  When $\hat\bx\not\in\Omega$ (e.g., $\hat\bx^b$ and $\hat\bx^c$), $\balpha = \bzero$ is the only feasible solution since $\hat\bx$ cannot be optimal for FOP($\balpha$) for any nonzero $\balpha$.
\end{example}

\subsection{Convex objective function imputation by~\citet{Keshavarz}}
\citet{Keshavarz} explicitly considered noise or modeling errors in imputing an objective function from a convex model so that the observed data from an unknown, complex system are as consistent as possible with the proposed model.  Given $K$ prespecified objectives denoted as a vector-valued function $\bff\in\mathbb{R}^{K}$ and a solution $\hat\bx$, they relax the KKT conditions and allow residuals, which are then minimized by solving the following problem:
\begin{subequations}\label{eq:kesh}
\begin{align}
\label{eq:kesh_a}\textrm{KES}(\hat\bx):\quad \underset{\balpha,\bsigma,\bpi,\bdelta,\bgamma,\brho}{\text{minimize}} & \quad \phi(\bdelta,\bgamma,\brho)\\
\label{eq:kesh_b}\text{subject to} & \quad \sum_{k=1}^K\alpha_k \nabla_{\!\!\bx} f_k(\hat\bx) + \sum_{l=1}^L\sigma_l \nabla_{\!\!\bx} g_l(\hat\bx) - \bA'\bpi = \bdelta,\\
\label{eq:kesh_c}& \quad \sigma_l g_l(\hat\bx) = \gamma_l, \quad l=1,\ldots,L,\\
\label{eq:kesh_d}& \quad \pi_j(\ba_j^{\prime}\hat\bx-b_j) = \rho_j, \quad j=1,\ldots,m,\\
\label{eq:kesh_e}&\quad \alpha_1 = 1,\\
\label{eq:kesh_f}&\quad \balpha\ge \bzero,\;\bsigma \ge \bzero,
\end{align}
\end{subequations}
where $\ba_j^{\prime}$ denotes the $j$-th row of $\bA$ and $\phi$ is a function of the residuals such that $\phi(\bdelta,\bgamma,\brho)=0$ if and only if $\bdelta=\bzero$, $\bgamma=\bzero$, and $\brho=\bzero$ (e.g., $\phi(\bdelta,\bgamma,\brho) = ||\bdelta||_2^2+||\bgamma||_2^2 + ||\brho||_2^2$).  Constraint~\eqref{eq:kesh_e} ensures that $\balpha=\bzero$ is not a feasible solution and serves to implicitly normalize the resulting weight vector.  The original version of formulation~\eqref{eq:kesh} in~\citet{Keshavarz} accommodated multiple input data points each with their own residuals, but we simply illustrate their method with a single data point $\hat\bx$.  The extension to multiple input data points is straightforward.  The original model also excluded the $\brho$ residual in~\eqref{eq:kesh_d} as the focus was on $\hat\bx \in \bX$, but we include it here to emphasize the applicability of the model to points $\hat\bx \not\in \bX$.   As noted in their paper, the choice of $\phi$ impacts the solution.  However, how the weights are normalized can also have a large impact, as we demonstrate in the example below and our computational results.
\begin{example}\label{ex:kesh}
We revisit Example~\ref{ex:iyengar} and apply formulation~\eqref{eq:kesh} to $\hat\bx^b$:
\begin{subequations}\label{eq:ex_kesh}
\begin{align}
\underset{\balpha,\sigma,\bdelta,\gamma}{\text{minimize}} & \quad \delta_1^2 + \delta_2^2 + \gamma^2\\
\text{subject to} & \quad 1.7(4\alpha_1 + \alpha_2) - 0.3\sigma = \delta_1,\\
& \quad 1.3(\alpha_1 + 4\alpha_2) - 0.7\sigma = \delta_2,\\
& \quad -0.42\sigma=\gamma,\\
&\quad \alpha_1 = 1,\\
&\quad \balpha\ge \bzero,\;\sigma \ge 0.
\end{align}
\end{subequations}
The optimal weight vector is $(\alpha^*_1, \alpha^*_2) = (1,0)$. The corresponding FOP($\balpha^*$) generates the optimal solution $\bx^*=(1.067,1.641)$.  Comparing $\bff(\hat\bx^b) = (13.250,9.650)$ to $\bff(\bx^*)=(7.244,11.910)$ we see one objective value decreases while the other increases.  Furthermore, if the normalization constraint is changed to $\alpha_2=1$, then the optimal weight vector changes to $\balpha^* = (0,1)$ and the new objective vector $\bff(\bx^*)=(11.910,7.244)$, which is completely opposite to the previous case.
\end{example}

\subsection{Generalized inverse multi-objective linear optimization by~\cite{chan_giop}}
\cite{chan_giop} considered the multi-objective linear optimization problem $\min\{\balpha'\bC\bx\;|\;\bA\bx=\bb,\bx\ge\bzero\}$ as the forward problem, where $\bC$ is a matrix composed of linear objectives $\bc'_k$ (one per row $k$).  In the corresponding inverse problem, the residual to be minimized was the duality gap:
\begin{subequations}\label{eq:chan}
\begin{align}
\quad\underset{\balpha,\bp,\epsilon}{\textrm{minimize}}& \quad \epsilon\\
\quad\textrm{subject to} & \quad \bA'\bp\le\bC'\balpha,\\
& \quad\balpha^\prime\bC\hat\bx=\epsilon\,\bb'\bp,\\
& \quad\|\balpha\|_1=1,\\
& \quad\balpha\ge\bzero,
\end{align}
\end{subequations}
where $\bp$ denotes the dual vector corresponding to the constraint $\bA\bx=\bb$ of the forward program. It was shown that this $\epsilon$-strong duality approach produces a weight vector $\balpha^*$ and a corresponding  forward solution $\bx^*\in\Omega(\balpha^*)$ such that $\bc_k'\bx^* / \bc_k'\hat\bx$ is equal for each objective $k$ for which $\alpha_k>0$. In the next section, we propose a new inverse model that generalizes \cite{chan_giop} to general convex multi-objective optimization problems by explicitly using the notion of trade-off preservation.

\section{Models}
In this section, we formalize the definition of trade-off preservation and develop an inverse convex multi-objective optimization model that determines a weight vector that preserves the trade-off encoded in a given solution. We then propose approximation approaches to speed up computation.  Finally, we draw connections between our inverse model and other models from the literature.

\subsection{Trade-off preservation}
When a given solution is not on the Pareto frontier, we propose that an inverse multi-objective optimization model should find a weight vector that generates a new solution that is Pareto optimal \emph{and} preserves the trade-off encoded in the given solution. With access only to the given solution, we can only infer the trade-off from its objective values. Thus, we characterize how a trade-off is preserved by examining the direction and length of the vector that connects the Pareto optimal solution and the given solution in objective space. 

We first define general trade-off preservation in terms of the component-wise relation between two solutions in objective space.
\begin{definition}[Generalized trade-off preservation]\label{def:PPP}
Given scalars $u_{k_1 k_2}$ for every pair of objectives $k_1$ and $k_2$, a solution $\bx^*$ preserves the trade-off encoded in $\hat\bx$ if the following equalities hold: 
\begin{equation}\label{eq:PPP}
u_{k_1k_2}\Big(f_{k_1}(\bx^*) - f_{k_1}(\hat\bx)\Big) = \Big(f_{k_2}(\bx^*) - f_{k_2}(\hat\bx)\Big), \quad \forall k_1, k_2\in\{1,\ldots,K\}.
\end{equation}
A weight vector $\balpha^*$ preserves the trade-off encoded in $\hat\bx$ if $\bx^* \in \Omega(\balpha^*)$.
\end{definition}
The quantity $u_{k_1k_2}$ is a problem-specific scaling factor that captures the relative difference between objectives $k_1$ and $k_2$. For example, if one unit of measurement for objective 1 is equivalent to 10 units for objective 2, then we set $u_{12}=10$ and $u_{21} = 1/10$. Then, by equating the adjusted amount of perturbation across all the objectives, the perturbed solution $\bx^*$ preserves the initial trade-offs encoded in the given solution $\hat\bx$. Naturally, the scaling factor should have the following properties: 
\begin{enumerate}[(i)]

\item $u_{kk}=1, \quad \forall k \in\{1,\ldots,K\},$

\item $u_{k_1 k_2} u_{k_2 k_1} = 1, \quad \forall k_1, k_2, \in\{1,\ldots,K\}$, and

\item $u_{k_1 k_2} u_{k_2 k_3} = u_{k_1 k_3}, \quad \forall k_1, k_2,  k_3 \in\{1,\ldots,K\}$.

\end{enumerate}
Next, we present two specific cases of trade-off preservation that are intuitive and have a direct connection with the concept of a duality gap. The first case considers two solutions whose objective values are component-wise proportional to be trade-off preserving, which is of practical use when the objectives are measured in different units whose values lie in ranges with substantial variation.  We refer to this type of trade-off preservation as \emph{relative trade-off preservation}.
\begin{definition}[Relative trade-off preservation]\label{def:PPP_rel}
A solution $\bx^*$ preserves the relative trade-off encoded in $\hat\bx$ if the following equalities hold for every pair of objectives $k_1$ and $k_2$:
\begin{equation}\label{eq:PPP_rel}
\frac{f_{k_1}(\bx^*) }{ f_{k_1}(\hat\bx) } = \frac{f_{k_2}(\bx^*)}{f_{k_2}(\hat\bx)}, \quad \forall k_1, k_2\in\{1,\ldots,K\}.
\end{equation}
A weight vector $\balpha^*$ preserves the relative trade-off encoded in $\hat\bx$ if $\bx^* \in \Omega(\balpha^*)$.
\end{definition}

Definition~\ref{def:PPP_rel} says that the objective values of $\bx^*$ are adjusted component-wise by the same relative amount from the objective values of $\hat\bx$.  Geometrically, this definition means that $\bff(\bx^*)$ lies on the line joining $\bff(\hat\bx)$ and the origin. Since our focus is to find a Pareto optimal solution that is trade-off preserving, we are interested in identifying where that line intersects the Pareto set, if at all. This concept is also intimately connected with duality. In particular, in the next subsection we will show that the weight vector $\balpha^*$ that satisfies the relative trade-off preservation leads to the minimum relative duality gap with respect to a given solution $\hat\bx\in\bX$. The relative trade-off preservation is in fact a special case of the general definition of trade-off preservation (i.e., Definition~\ref{def:PPP}) if $u_{k_1 k_2} = f_{k_2}(\hat\bx) / f_{k_1}(\hat\bx), \forall k_1, k_2\in\{1,\ldots,K\}$. Note that these specific scaling factors also satisfy the properties described above.

Similarly, we define an absolute version of trade-off preservation, where the given solution's objective values are perturbed by the same absolute amount. Such an approach may be preferred when the objective values are measured in the same units over the same range.
\begin{definition}[Absolute trade-off preservation]\label{def:PPP_abs}
A solution $\bx^*$ preserves the absolute trade-off encoded in $\hat\bx$ if the following equality holds for every pair of objectives $k_1$ and $k_2$:
\begin{equation}\label{eq:PPP_abs}
f_{k_1}(\bx^*) - f_{k_1}(\hat\bx) = f_{k_2}(\bx^*) - f_{k_2}(\hat\bx), \quad \forall k_1, k_2\in\{1,\ldots,K\}.
\end{equation}
A weight vector $\balpha^*$ preserves the absolute trade-off encoded in $\hat\bx$ if $\bx^* \in \Omega(\balpha^*)$.
\end{definition}
Absolute trade-off preservation is a special case of general trade-off preservation in Definition~\ref{def:PPP} where $u_{k_1k_2} = 1, \forall k_1, k_2\in\{1,\ldots,K\}$. Similar to the relative case, we will show in the next subsection that absolute trade-off preservation has a direct relationship with minimizing the absolute duality gap in an inverse optimization model.

\subsection{Inverse optimization models with trade-off preservation}\label{sec:IOP}
In this subsection, we propose trade-off-preserving inverse multi-objective convex optimization models. We start with the model that takes into account the general definition of trade-off preservation and then present the models for the relative and absolute specializations.

\subsubsection{Inverse optimization model with general trade-off preservation}

We formulate a convex optimization model that determines a weight vector $\balpha^*$ and corresponding optimal solution $\bx^*\in\Omega(\balpha^*)$ that satisfy equation~\eqref{eq:PPP}. To do so, we first define $\mu_k := u_{\tilde{k} k}, k=1,\ldots,K,$ where $\tilde{k} \in \{1,\ldots,K\}$ is an arbitrarily chosen reference objective. Given a set of objectives and an input solution $\hat\bx$, the following formulation, which we call the inverse optimization problem (IOP), provides a necessary and sufficient condition for determining whether there exists a $\bx^* \in \Omega$ that satisfies general trade-off preservation:
\begin{subequations}\label{eq:iop}
\begin{align}
\label{eq:iop1} \text{IOP}(\hat\bx) : \underset{\epsilon,\bx}{\text{minimize}} & \quad \epsilon\\
\label{eq:iop2} \text{subject to} & \quad \mu_k \epsilon\ge f_k(\bx) - f_k(\hat\bx),\quad k = 1,\ldots,K,\\
\label{eq:iop3} & \quad g_l(\bx) \le 0,\quad l=1,\ldots,L,\\
\label{eq:iop4} & \quad \bA\bx = \bb.
\end{align}
\end{subequations}
\begin{theorem}\label{prop:exists_xinOmega}
Let $(\epsilon^*,\bx^*)$ be an optimal solution to problem~\eqref{eq:iop} and $\balpha^*$ be a vector of optimal Lagrange multipliers associated with constraints~\eqref{eq:iop2}. Then,
\begin{itemize}
\item[(a)] There exists a solution in $\Omega$ that satisfies general trade-off preservation for $\hat\bx$ if and only if there exists $(\epsilon^*,\bx^*)$ satisfying the first $K$ constraints with equality.
\item[(b)] $\bx^*\in\Omega(\balpha^*)$.
\end{itemize}
\end{theorem}
\begin{proof} (a) ($\Leftarrow$) Let $(\epsilon^*,\bx^*)$ be an optimal solution to problem~\eqref{eq:iop} with the first $K$ constraints being tight. Then
\begin{subequations}
\begin{align}
&\quad \epsilon^* = \frac{(f_k(\bx^*) - f_k(\hat\bx))}{\mu_{k}},\quad \forall k = 1,\ldots,K\\
\Rightarrow& \quad \frac{f_{k_1}(\bx^*) - f_{k_1}(\hat\bx)}{u_{ \tilde{k} k_1}} = \frac{f_{k_2}(\bx^*) - f_{k_2}(\hat\bx)}{u_{ \tilde{k} k_2}},\quad \forall k_1,k_2 = 1,\ldots,K\\
\Rightarrow& \quad \frac{u_{\tilde{k} k_2 }}{u_{\tilde{k} k_1 }}\Big(f_{k_1}(\bx^*) - f_{k_1}(\hat\bx)\Big) = \Big(f_{k_2}(\bx^*) - f_{k_2}(\hat\bx)\Big),\quad \forall k_1,k_2 = 1,\ldots,K\\
\Rightarrow& \quad u_{k_1 \tilde{k} } u_{ \tilde{k} k_2} \Big(f_{k_1}(\bx^*) - f_{k_1}(\hat\bx)\Big) = \Big(f_{k_2}(\bx^*) - f_{k_2}(\hat\bx)\Big),\quad \forall k_1,k_2 = 1,\ldots,K\\
\Rightarrow& \quad u_{k_1 k_2}\Big(f_{k_1}(\bx^*) - f_{k_1}(\hat\bx)\Big) = \Big(f_{k_2}(\bx^*) - f_{k_2}(\hat\bx)\Big),\quad \forall k_1,k_2 = 1,\ldots,K,
\end{align}
\end{subequations}
which implies that $\bx^*$ satisfies trade-off preservation. All that remains is to prove $\bx^* \in \Omega$. Constraints~\eqref{eq:iop2} and~\eqref{eq:iop3} imply that $\bx^*\in\bX$. Let $\balpha \ge \bzero$, $\bsigma\ge\bzero$, and $\bpi$ be the Lagrange multipliers associated with the first, second, and third sets of constraints of problem~\eqref{eq:iop}, respectively. Consider the Lagrangian associated with problem~\eqref{eq:iop}:
\begin{equation}\label{eq:iop_largange}
L(\balpha,\bsigma,\bpi,\epsilon,\bx) = \epsilon + \sum_{k=1}^K\alpha_k (f_k(\bx)- f_k(\hat\bx) - \mu_k\epsilon) + \sum_{l=1}^L\sigma_l g_l(\bx) + \bpi'(\bb-\bA\bx).
\end{equation}
A solution ($\epsilon^*,\bx^*$) is optimal for problem~\eqref{eq:iop} if and only if there exists $(\balpha,\bsigma,\bpi)\in\mathbb{R}_+^K \times \mathbb{R}_+^L \times \mathbb{R}^m$ that satisfies the following system of equations:
\begin{subequations}\label{eq:iop_KKT}
\begin{align}
\label{eq:iop_KKT1} & 1 - \sum_{k=1}^K \mu_k \alpha_k  = 0, \quad (\nabla_{\!\!\epsilon} L(\balpha,\bsigma,\bpi,\epsilon^*,\bx^*) = 0)\\
\label{eq:iop_KKT2}&\sum_{k=1}^K\alpha_k \nabla_{\!\!\bx} f_k(\bx^*) + \sum_{l=1}^L\sigma_l \nabla_{\!\!\bx} g_l(\bx^*) - \bA'\bpi = \bzero, \quad (\nabla_{\!\!\bx} L(\balpha,\bsigma,\bpi,\epsilon^*,\bx^*) = 0)\\
\label{eq:iop_KKT3} &\alpha_k (f_k(\bx^*)- f_k(\hat\bx) - \mu_k\epsilon^*) = 0, \quad k=1,\ldots,K,\\
\label{eq:iop_KKT4} &\sigma_l g_l(\bx^*) = 0, \quad l=1,\ldots,L.
\end{align}
\end{subequations}
Equations~\eqref{eq:iop_KKT2} and~\eqref{eq:iop_KKT4} form the KKT conditions for the original forward optimization problem~\eqref{eq:fop}.  From \eqref{eq:iop_KKT1}, $\balpha$ is a nonzero vector.   Thus, if a solution $\bx^*$ is optimal for problem~\eqref{eq:iop}, $\bx^*\in\Omega$.

($\Rightarrow$)
Let $\bx^*\in\Omega$ satisfy the definition of general trade-off preservation, i.e., for any two objectives $k_1, k_2 \in \{1,\ldots,K\}$, $u_{k_1 k_2}\Big(f_{k_1}(\bx^*) - f_{k_1}(\hat\bx)\Big) = \Big(f_{k_2}(\bx^*) - f_{k_2}(\hat\bx)\Big)$. The KKT conditions for the forward problem imply that there exists $(\balpha^*, \bsigma^*, \bpi^*)$ that satisfies \eqref{eq:iop_KKT2} and  \eqref{eq:iop_KKT4} with at least one $k$ for which $\alpha_k^*>0$. What remains is to show that \eqref{eq:iop_KKT1} and  \eqref{eq:iop_KKT3}  are also satisfied.  Since $u_{k_1 k_2}\Big(f_{k_1}(\bx^*) - f_{k_1}(\hat\bx)\Big) = \Big(f_{k_2}(\bx^*) - f_{k_2}(\hat\bx)\Big), k_1,k_2 = 1,\ldots,K,$ the following is true for some objective $\tilde k$:
\begin{subequations}
\begin{align}
& \quad u_{k_1 \tilde{k} } u_{ \tilde{k} k_2} \Big(f_{k_1}(\bx^*) - f_{k_1}(\hat\bx)\Big) = \Big(f_{k_2}(\bx^*) - f_{k_2}(\hat\bx)\Big),\quad k_1,k_2 = 1,\ldots,K\\
\Rightarrow& \quad \frac{f_{k_1}(\bx^*) - f_{k_1}(\hat\bx)}{u_{ \tilde{k} k_1}} = \frac{f_{k_2}(\bx^*) - f_{k_2}(\hat\bx)}{u_{ \tilde{k} k_2}},\quad k_1,k_2 = 1,\ldots,K\\
\Rightarrow&\quad \exists\epsilon^* \text{ s.t. } \epsilon^* = \frac{(f_k(\bx^*) - f_k(\hat\bx))}{u_{ \tilde{k} k}},\quad k = 1,\ldots,K\\
\Rightarrow&\quad \mu_k \epsilon^* = f_k(\bx) - f_k(\hat\bx),\quad \text{where $\mu_k =u_{\tilde{k} k}$}, k = 1,\ldots,K,
\end{align}
\end{subequations}
which implies that \eqref{eq:iop_KKT3} is satisfied. Finally, equation~\eqref{eq:iop_KKT1} can be satisfied through a re-scaling of $\balpha^*$, which is possible since $\balpha^*$ is not identically zero.  Thus, $(\epsilon^*,\bx^*)$ is optimal for~\eqref{eq:iop}.

(b) Since $\bx^*$  and $\balpha^*$ satisfy \eqref{eq:iop_KKT2} and \eqref{eq:iop_KKT4}, which are the KKT conditions for the forward problem~\eqref{eq:fop}, $\bx^* \in \Omega(\balpha^*)$. $\square$
\end{proof}

Theorem~\ref{prop:exists_xinOmega} suggests that solving IOP$(\hat\bx)$ can simultaneously identify, if they exist, a trade-off preserving solution $\bx^* \in \Omega$ and a corresponding optimal weight vector $\balpha^*$. Theorem~\ref{prop:exists_xinOmega} also implies that if there does not exist an optimal solution to problem~\eqref{eq:iop} that satisfies the first $K$ constraints with equality, then there is no trade-off preserving solution in $\Omega$.

Note that the choice of the reference objective $\tilde{k}$ in the definition of $\mu_k$ simply scales the optimal objective value for problem~\eqref{eq:iop}, without changing the feasible region of~\eqref{eq:iop}. Because the functions $f_k(\cdot)$ and $g_l(\cdot)$ are convex functions, formulation~\eqref{eq:iop} is a convex optimization problem.

\subsubsection{Inverse optimization model with relative trade-off preservation}

The following inverse model determines a weight vector $\balpha^*$ and $\bx^* \in \text{FOP}(\balpha^*)$ that satisfy relative trade-off preservation given a solution $\hat\bx$:
\begin{subequations}\label{eq:iop_rel}
\begin{align}
\label{eq:iop_rel1} \text{IOP}_r(\hat\bx) : \quad \underset{\epsilon,\bx}{\text{minimize}} & \quad \epsilon\\
\label{eq:iop_rel2} \text{subject to} & \quad \epsilon f_k(\hat\bx)\ge f_k(\bx),\quad k = 1,\ldots,K,\\
\label{eq:iop_rel3} & \quad g_l(\bx) \le 0,\quad l=1,\ldots,L,\\
\label{eq:iop_rel4} & \quad \bA\bx = \bb.
\end{align}
\end{subequations}

Theorem~\ref{prop:exists_xinOmega} holds for problem~\eqref{eq:iop_rel} because \eqref{eq:iop_rel} is a special case of the general inverse model~\eqref{eq:iop} where $\mu_k = u_{\tilde{k} k} = f_{ k}(\hat\bx) / f_{\tilde k}(\hat\bx), \forall k \in\{1,\ldots,K\}$. The derivation is straightforward and omitted. As mentioned previously, the definition of relative trade-off preservation provides a simple and intuitive connection with minimizing relative duality  gap.
\begin{proposition}\label{prop:minrelgap}
Given $\hat\bx\in\bX$, let $(\epsilon^*,\bx^*)$ be an optimal solution to $\text{\emph{IOP}}_r(\hat\bx)$. Then $1/\epsilon^*$ is the minimum relative duality gap with respect to $\hat\bx$. 
\end{proposition}
\begin{proof}
Consider the Lagrangian associated with $\text{IOP}_r(\hat\bx)$ (i.e., \eqref{eq:iop_rel}):
\begin{equation}\label{eq:iop_rel_largangian}
L(\balpha,\bsigma,\bpi,\epsilon,\bx)  =  \epsilon + \sum_{k=1}^K\alpha_k (f_k(\bx)- \epsilon f_k(\hat\bx)) + \sum_{l=1}^L\sigma_l g_l(\bx) + \bpi'(\bb-\bA\bx),
\end{equation}
and the corresponding dual problem: $\underset{\balpha,\bdelta,\bpi}{\max}\; \underset{\epsilon,\bx}{\min}\; L(\balpha,\bsigma,\bpi,\epsilon,\bx)$. Let $(\epsilon^*,\bx^*)$ be an optimal solution to \eqref{eq:iop_rel}. Then we have $\sigma_l g_l(\bx^*) = 0, \forall l = 1,\ldots,L$ and $\bA\bx^* = \bb$, and need $\balpha$ to satisfy $\sum_{k=1}^K\alpha_k  f_k(\hat\bx) = 1$ (i.e., $\nabla_{\epsilon}L = 0$), which cancel out all the terms but $\sum_{k=1}^K\alpha_k f_k(\bx)$ in \eqref{eq:iop_rel_largangian}. Thus, 
\begin{equation}\label{eq:eps_mingap}
\epsilon^* = \max\; \{\balpha' \bff(\bx)\,|\, \balpha' \bff(\hat\bx) = 1, \balpha \ge \bzero, \bx \in \Omega(\balpha)\},
\end{equation}
and therefore 
\begin{subequations}\label{eq:eps_mingap2}
\begin{align}
\label{eq:eps_mingap2_1} 1/\epsilon^* & = \min\; \{1/\balpha' \bff(\bx) \,|\, \balpha' \bff(\hat\bx) = 1, \balpha \ge \bzero, \bx \in \Omega(\balpha) \}\\
\label{eq:eps_mingap2_2} & = \min\; \{\balpha' \bff(\hat\bx)/\balpha' \bff(\bx) \,|\, \balpha' \bff(\hat\bx) = 1, \balpha \ge \bzero, \bx \in \Omega(\balpha) \}.
\end{align}
\end{subequations}
Next, consider a problem that determines the minimum relative duality gap with respect to $\hat\bx$:
\begin{subequations}\label{eq:eps_mingap3}
\begin{align}
\underset{\balpha,\bx}{\text{minimize}} & \quad \balpha' \bff(\hat\bx)/\balpha' \bff(\bx) \\
\text{subject to} & \quad \bx \in \Omega(\balpha), \\
& \quad \balpha \ge \bzero.
\end{align}
\end{subequations}
For any feasible solution $(\tilde\balpha,\tilde\bx)$ for~\eqref{eq:eps_mingap3}, we can find a scaled vector $\balpha^* = \frac{1}{\tilde\balpha'\bff(\hat\bx)}\tilde\balpha$ such that ${\balpha^*}'\bff(\hat\bx) = 1$, $\tilde\bx \in \Omega(\balpha^*)$, and thus $(\balpha^*, \tilde\bx)$ is feasible for \eqref{eq:eps_mingap2_2} with the same objective value. Therefore, the optimal objective value for \eqref{eq:eps_mingap3} coincides with $1/\bepsilon^*$, as desired. $\square$
\end{proof}

Note that the complementary slackness conditions associated with problem~\eqref{eq:iop_rel} provide a sufficient condition for when the relative trade-off preservation is satisfied, which is given in the next result. The proof is straightforward and omitted.
\begin{corollary}\label{prop:relationship}
Let $\bx^*$ be an optimal solution to \emph{IOP}$_r(\hat\bx)$ (i.e., problem~\eqref{eq:iop_rel}) and $\balpha^*$ be a vector of optimal Lagrange multipliers associated with constraints \eqref{eq:iop_rel2}.  If $\alpha_k^* > 0$ for all $k=1,\ldots,K$, then $\bx^*$ satisfies relative trade-off preservation for $\hat\bx$.
\end{corollary}
The following example illustrates that if there is at least one $k$ such that $\alpha_k^*=0$, then there may or may not be a solution in $\Omega$ that preserves the relative trade-off.
\begin{example}\label{ex:iop}\normalfont
Consider Example~\ref{ex:iyengar} again and two new solutions $\hat\bx^d = (1.725,1.121)$ and $\hat\bx^e = (1.789,1.096)$.  Figure~\ref{fig:example3_1} shows the feasible region and the Pareto set in objective space (i.e., $\bff(\bX)$ and $\bff(\Omega)$, respectively), along with the two points $\bff(\hat\bx^d)=(13.160,8.004)$ and $\bff(\hat\bx^e)=(14.000,8.004)$.  Solving IOP$_r$($\hat\bx^d$) and IOP$_r$($\hat\bx^e$) both return the same optimal weight vector $\balpha^* = (0,1)$ and the same forward optimal solution $\bx^* = \bar\bx^d \in \Omega(\balpha^*)$.  It is clear from Figure~\ref{fig:example3_1} that $\bar\bx^d$ preserves the trade-off preference of $\hat\bx^d$ but not $\hat\bx^e$.  Because $\bff(\bar\bx^e) \not\in \bff(\Omega)$, there is no non-zero weight vector that makes $\bar\bx^e$ optimal to the forward problem.  The case of $\hat\bx^d$ also illustrates the possibility of degeneracy, i.e., $\epsilon^* = f_1(\bx^*)/f_1(\hat\bx^d)$ while $\alpha_1=0$.
\begin{figure}\centering
\includegraphics[width=80mm]{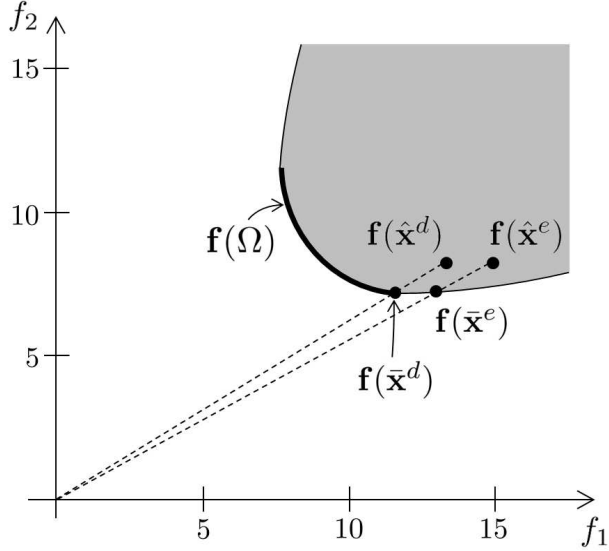}
\caption{Illustration of Example~\ref{ex:iop}}\label{fig:example3_1}
\end{figure}
\end{example}

Example~\ref{ex:iop} illustrates the intuitive fact that when a given $\hat\bx$ is sufficiently ``inferior'' with respect to a particular objective $k$, the optimal weight vector determined from inverse optimization will return $\alpha^*_k = 0$.  In this case the relative trade-off is not perfectly preserved.  However, the objectives with non-zero weights will still satisfy relative trade-off preservation.

\subsubsection{Inverse optimization model with absolute trade-off preservation}
Similar to problem~\eqref{eq:iop_rel} for relative trade-off preservation, a necessary and sufficient condition for determining whether there exists a solution $\bx^* \in \Omega$ that satisfies absolute trade-off preservation for $\hat\bx$ can be established by the following convex optimization model:
\begin{subequations}\label{eq:iop_abs}
\begin{align}
\label{eq:iop_abs1} \text{IOP}_a(\hat\bx) : \underset{\epsilon,\bx}{\text{minimize}} & \quad \epsilon\\
\label{eq:iop_abs2} \text{subject to} & \quad \epsilon\ge f_k(\bx) - f_k(\hat\bx),\quad k = 1,\ldots,K,\\
\label{eq:iop_abs3} & \quad g_l(\bx) \le 0,\quad l=1,\ldots,L,\\
\label{eq:iop_abs4} & \quad \bA\bx = \bb.
\end{align}
\end{subequations}
Similar to the relative model, Theorem~\ref{prop:exists_xinOmega} holds for~\eqref{eq:iop_abs} because it is a special case of the general model~\eqref{eq:iop} where $u_{k_1 k_2}= 1$ for all $k_1, k_2 \in \{1,\ldots,K\}$. Similar to Proposition~\ref{prop:minrelgap}, Proposition~\ref{prop:minabsgap} suggests that a weight vector that produces a solution $\bx^*$ that satisfies absolute trade-off preservation for $\hat\bx$ minimizes the absolute optimality gap with respect to $\hat\bx\in\bX$. The proof of this result involves minor modifications to the proof of Proposition~\ref{prop:minrelgap} and is omitted.
\begin{proposition}\label{prop:minabsgap}
Given $\hat\bx\in\bX$, let $(\epsilon^*,\bx^*)$ be an optimal solution to $\text{\emph{IOP}}_a(\hat\bx)$. Then $-\epsilon^*$ is the minimum absolute duality gap with respect to $\hat\bx$. That is, $-\epsilon^* = \min \{\balpha'\bff(\hat\bx)-\balpha'\bff(\bx)\,|\,\bx\in\Omega(\balpha), \balpha\ge\bzero\} \ge 0$.
\end{proposition}

Finally, the following result provides a sufficient condition for when absolute trade-off preservation is satisfied. Again, the proof is straightforward and omitted.
\begin{corollary}\label{prop:relationship_abs}
Let $\bx^*$ be an optimal solution to problem~\eqref{eq:iop_abs} and $\balpha^*$ be a vector of optimal Lagrange multipliers associated with constraints~\eqref{eq:iop_abs2}.  If $\alpha_k^* > 0$ for all $k=1,\ldots,K$, then $\bx^*$ satisfies absolute trade-off preservation for $\hat\bx$.
\end{corollary}

\subsection{Linearized inverse optimization models with trade-off preservation}
In this section we formulate a linear approximation to IOP($\hat\bx$) (i.e., formulation~\eqref{eq:iop}). The reasons for presenting the linearization are two-fold. First, the linearized model helps elucidate connections with other inverse models in the literature, as we will discuss in the next section. Second, as we will demonstrate in the computational results, the linear approximation approach is computationally efficient to implement and may generate a solution that is close to the optimal solution from the exact inverse model, either as a stand-alone model or as part of a specialized solution methodology for the exact model.

The linearized inverse optimization problem (LIOP) for the general model~\eqref{eq:iop} can be formulated by linearizing the first two sets of constraints of~\eqref{eq:iop} around some point $\tilde\bx$:
\begin{subequations}\label{eq:liop}
\begin{align}
\label{eq:liop1} \text{LIOP}(\hat\bx,\tilde\bx) : \underset{\epsilon,\bx}{\text{minimize}} & \quad \epsilon\\
\label{eq:liop2} \text{subject to} & \quad \mu_k\epsilon\ge f_k(\tilde\bx) + \nabla_{\!\!\bx}f_k(\tilde\bx)'(\bx-\tilde\bx) - f_k(\hat\bx),\quad k = 1,\ldots,K,\\
\label{eq:liop3} & \quad g_l(\tilde\bx) + \nabla_{\!\!\bx}g_l(\tilde\bx)(\bx-\tilde\bx) \le 0,\quad l=1,\ldots,L,\\
\label{eq:liop4} & \quad \bA\bx = \bb.
\end{align}
\end{subequations}
Since $f_k$ and $g_l$ are convex, it is straightforward to see that formulation~\eqref{eq:liop} forms an outer approximation -- and thus provides a lower bound -- to formulation~\eqref{eq:iop}.  If $\tilde\bx = \bx^*$, where $\bx^*$ is an optimal solution to problem~\eqref{eq:iop}, then an optimal Lagrange multiplier for constraints~\eqref{eq:liop2} is also optimal for~\eqref{eq:iop}.  Unfortunately, $\bx^*$ is not known \emph{a priori}.  One option is simply to solve~\eqref{eq:liop} with $\tilde\bx=\hat\bx$ (i.e., solve LIOP$(\hat\bx,\hat\bx)$, which for brevity we refer to as LIOP$(\hat\bx)$).  Note that linearization may render problem~\eqref{eq:liop} unbounded.  In this case, a trust region $\bx\in[\hat\bx-\kappa\beee,\hat\bx+\kappa\beee]$ for some $\kappa$ can be added, where $\beee$ is the vector of ones \citep{nlpbook}. 

Building on this idea, since~\eqref{eq:liop} is a first-order approximation of~\eqref{eq:iop}, we can employ the well-known successive linear programming (SLP) algorithm~\citep{nonlin_alg2,nlpbook} to solve formulation~\eqref{eq:iop}.  The algorithm repeatedly solves the linear problem~\eqref{eq:liop} using a trust region approach, generating an optimal solution $\bx^*_i$ in iteration $i$ that is used as the input vector $\tilde\bx$ in iteration $i+1$.  In Section~\ref{sec:results}, we implement the SLP algorithm based on~\citet{nonlin_alg2} -- we refer the reader to that paper for details of the algorithm and a proof of convergence.

\subsection{Relationships between inverse optimization models}
Focusing on the general model, if $\hat\bx \in \Omega$, then an optimal $\balpha^*$ derived from one of IOP($\hat\bx$), LIOP($\hat\bx$), or KES($\hat\bx$) is optimal for all of them. If $f_k, \, k = 1, \ldots, K$ and $g_l, \, l = 1, \ldots, L$ are linear, then IOP($\hat\bx$) and LIOP($\hat\bx$) are equivalent; particularly IOP($\hat\bx$) with the relative and absolute trade-off preservation specializes to models of~\cite{chan_giop} .   The linearization described in the previous section provides a bridge between our inverse formulation~\eqref{eq:iop} and the model of~\citet{Keshavarz}.  The next result draws an equivalence between LIOP($\hat\bx$) and KES($\hat\bx$), thus illustrating how the model of~\citet{Keshavarz} can be modified to take into account trade-off preservation.
\begin{proposition}\label{thm:equivalence}
An optimal weight vector solution to \emph{KES}$(\hat\bx)$ is identical to an optimal weight vector solution to the dual of \emph{LIOP}$(\hat\bx)$, if $\bdelta=\bzero, \phi(\bdelta,\bgamma,\brho) = -\bgamma'\beee + \brho'\beee$, and $\balpha$ is normalized by the constraint $\sum_{k=1}^K \mu_k \alpha_k = 1$ instead of $\alpha_1 = 1$.
\end{proposition}
\begin{proof}
Consider the dual of LIOP$(\hat\bx)$:
\begin{subequations}\label{eq:liop_dual}
\begin{align}
\label{eq:liop_dual_a} \underset{\balpha,\bsigma,\bpi}{\text{maximize}}&\quad
\bb'\bpi - \sum_{k=1}^K \alpha_k {\hat\bx}'\nabla_{\!\!\bx} f_k(\hat\bx) + \sum_{l=1}^L\sigma_l (g_l(\hat\bx)- {\hat\bx}'\nabla_{\!\!\bx} g_l(\hat\bx))\\
\label{eq:liop_dual_b} \text{subject to} & \quad \sum_{k=1}^K \mu_k \alpha_k = 1,\\
\label{eq:liop_dual_c} & \quad \sum_{k=1}^K\alpha_k \nabla_{\!\!\bx} f_k(\hat\bx) + \sum_{l=1}^L\sigma_l \nabla_{\!\!\bx} g_l(\hat\bx) - \bA'\bpi = \bzero,\\
\label{eq:liop_dual_d} & \quad \balpha \ge \bzero,\; \bsigma \ge \bzero.
\end{align}
\end{subequations}
We take the inner product of constraint~\eqref{eq:liop_dual_c} with $\hat\bx$ to get
\begin{equation}\label{eq:nabla_x}
\sum_{k=1}^K\alpha_k {\hat\bx}'\nabla_{\!\!\bx} f_k(\hat\bx) + \sum_{l=1}^L\sigma_l {\hat\bx}'\nabla_{\!\!\bx} g_l(\hat\bx) - \bpi'\bA\hat\bx = 0.
\end{equation}
By adding the left-hand side of equation~\eqref{eq:nabla_x} to the objective function of problem~\eqref{eq:liop_dual}, we obtain the following problem:
\begin{subequations}\label{eq:liop_dual3}
\begin{align}
\label{eq:liop_dual3_1} \underset{\balpha,\bsigma,\bpi}{\text{maximize}}& \quad \sum_{l=1}^L\sigma_l g_l(\hat\bx) - \bpi'(\bA\hat\bx-\bb),\\
\label{eq:liop_dual3_2} \text{subject to} & \quad \sum_{k=1}^K \mu_k \alpha_k = 1,\\
\label{eq:liop_dual3_3} & \quad \sum_{k=1}^K\alpha_k \nabla_{\!\!\bx} f_k(\hat\bx) + \sum_{l=1}^L\sigma_l \nabla_{\!\!\bx} g_l(\hat\bx) - \bA'\bpi = \bzero,\\
\label{eq:liop_dual3_4} & \quad \balpha \ge \bzero,\; \bsigma \ge \bzero.
\end{align}
\end{subequations}
KES($\hat\bx$) with $\bdelta=\bzero, \phi(\bdelta,\bgamma,\brho) = -\bgamma'\beee + \brho'\beee$, and $\balpha$ normalized by the constraint $\sum_{k=1}^K \mu_k \alpha_k = 1$ is
\begin{subequations}\label{eq:kesh_modified}
\begin{align}
\underset{\balpha,\bsigma,\bpi,\bgamma,\brho}{\text{minimize}} & \quad -\bgamma'\beee + \brho'\beee\\
\text{subject to} & \quad \sum_{k=1}^K\alpha_k \nabla_{\!\!\bx} f_k(\hat\bx) + \sum_{l=1}^L\sigma_l \nabla_{\!\!\bx} g_l(\hat\bx) - \bA'\bpi = \bzero,\\
& \quad \sigma_l g_l(\hat\bx) = \gamma_l, \quad l=1,\ldots,L,\\
& \quad \pi_j(\ba_j^{\prime}\hat\bx-b_j) = \rho_j, \quad j=1,\ldots,m,\\
&\quad \sum_{k=1}^K \mu_k \alpha_k = 1,\\
&\quad \balpha\ge \bzero,\;\bsigma \ge \bzero,
\end{align}
\end{subequations}
which is equivalent to problem~\eqref{eq:liop_dual3}. $\square$
\end{proof}

Note that if $\hat\bx\in\bX$, the form of $\phi(\bdelta,\bgamma,\brho)$ in the statement of  Proposition~\ref{thm:equivalence} can be replaced with $\phi(\bdelta,\bgamma,\brho)=||\bgamma||_1+||\brho||_1$.  Proposition~\ref{thm:equivalence} states that by specifying $\bdelta$, the function $\phi(\bdelta,\bgamma,\brho)$, and the normalization constraint in KES$(\hat\bx)$, a decision maker can preserve the trade-off encoded in $\hat\bx$ in the inverse optimization process. The following result shows how KES($\hat\bx$), which was initially proposed for single-objective problems, can be interpreted in the presence of multiple objectives.

\begin{corollary}\label{thm:equivalence_kes}
The linear approximation of the general inverse model~\eqref{eq:liop} with $\mu_1 = 1$ and $\mu_k=0$ for $k=2,\ldots,K$ is equivalent to KES($\hat\bx$).
\end{corollary}
Corollary~\ref{thm:equivalence_kes} indicates that KES($\hat\bx$) is an approximate version of the general inverse model~\eqref{eq:iop} where $\mu_1 = 1$ and $\mu_k=0$ for $k=2,\ldots,K$. That is, applying KES($\hat\bx$) to multi-objective optimization problems puts its entire emphasis on the adjustment of the first objective and in general does not take into account the trade-off across multiple objectives. This observation is in fact reinforced by our computational results in the next section. 

\section{Computational results}\label{sec:results}

In this section, we use data from prostate cancer radiation therapy and perform a retrospective study to demonstrate several inverse models presented in this paper. We use historical prostate cancer treatments as input into the inverse models to derive objective function weights that would approximately reproduce the treatments (via forward optimization).  We compare weights generated by the different inverse models and examine how the clinical criteria are affected by the weights. While the experiments constitute a retrospective analysis with historical treatment data, we also comment on clinical implications of our results.

\subsection{Background and motivation}
Radiation therapy is one of the primary methods to treat many cancers including prostate cancer~\citep{Foroudi}.  In prostate cancer treatment, the primary target is the prostate, while the healthy organs to be avoided include the rectum, bladder, and femoral heads (tops of the femur bones). Beams of radiation are delivered to the patient targeting a tumor at their intersection.  The patient's anatomy is discretized into small volume elements called voxels.  A radiation beam can be modeled as a set of small beamlets, whose intensities are optimized.  A radiation therapy treatment planning problem is typically a multi-objective optimization problem and the traditional approach to solving these problems clinically is through a trial-and-error approach to weight selection.

In response to the time consuming, trial-and-error approach to determining the weight values, knowledge-based treatment planning is an increasingly popular approach. Such an approach uses a database of historical treatments to derive a treatment plan for a \emph{de novo} patient (e.g.,~\cite{geom2},~\cite{geom3},~\cite{geom4},~\cite{lee_geom}, and~\cite{bout_geom}).  The idea is that planning effort may be reduced if we can leverage the vast knowledge and experience accumulated through creating clinically accepted plans for similar patients in the past.  Naturally, the historical treatments implicitly encode information on accepted trade-offs between conflicting clinical criteria (i.e., objectives).  Thus, any knowledge-based effort for determining the weight values should consider maintaining the clinical metrics (i.e., objective values) achieved by the historical treatments. 

In the computational results below, we apply different inverse optimization models proposed in this paper so as to impute weight values for a convex quadratic treatment planning problem and demonstrate to what extent the clinical metrics are preserved by each of the models. Specifically, for each historical treatment, which may not be mathematically exactly optimal for the underlying treatment planning problem, the inverse models determine sets of weight values that can reproduce the plan as closely as possible. We use relative trade-off preservation in this application because the clinical metrics, while all measured in terms of dose, may take on values in largely varying ranges (e.g., max dose values in the 70's and penalties for an organ exceeding a dose threshold in the single digits). We then show that weight values that approximate the given treatment under the consideration of trade-off preservation indeed generate a new treatment plan that is considered most similar clinically.

\subsection{Forward optimization problem}
Let $\bx\in\mathbb{R}^n_+$ be a vector of beamlet intensities and $\bD_k\in\mathbb{R}^{m_k \times n}_+$ be a matrix that quantifies the dose deposited to each voxel in structure $k$ from unit intensity of each beamlet, where $m_k$ denotes the number of voxels in structure $k$ and $n$ denotes the number of beamlets.  Let $\mathcal{T}$ index the tumor.  The vector $\bD_k\bx\in\mathbb{R}^{m_k}_+$ quantifies the dose to every voxel in structure $k$.  The radiation therapy treatment planning problem is a forward optimization problem.  We present a simple forward model below.  With each healthy organ $k$, we associate an objective function $f_k(\bx) = ||(\bD_k\bx-\theta_k\beee)_+||_2^2$, where $\theta_k$ denotes a dose threshold for structure $k$, and $(\cdot)_+$ denotes a vector with the operator $\max\{0,\cdot\}$ applied to each component.
\begin{subequations}\label{eq:fop_IMRT}
\begin{align}
\label{eq:fop_IMRT_a}\underset{\bx}{\text{minimize}}       \quad & \sum_{k=1}^K\alpha_k f_k(\bx),\\
\label{eq:fop_IMRT_b}\text{subject to}  \quad & \bl_k \le \bD_k\bx \le \bu_k, \quad k \in \{1,\ldots,K\}\cup
\{\mathcal{T}\},\\
 \label{eq:fop_IMRT_c}& \bzero \le \bx \le \beta\frac{\beee'\bx}{n}.
\end{align}
\end{subequations}
The objective penalizes delivering dose above a certain threshold for structure $k$.  Formulation~\eqref{eq:fop_IMRT} can be written as a convex program by rewriting $f_k(\bx)$ as $f_k(\bz) = ||\bz||_2^2$, where $\bz$ is an auxiliary decision vector satisfying $\bz \ge \bD_k\bx-\theta_k\beee$ and $\bz \ge \bzero$.  We use hard constraints~\eqref{eq:fop_IMRT_b} to bound the lower and upper doses to organ $k$ and the tumor $\mathcal{T}$.  Constraint~\eqref{eq:fop_IMRT_c} is a stylized approach to discouraging individual beamlets from significantly exceeding the mean beamlet intensity.  In the experiments below, five healthy organs-at-risk (OAR) comprise the objective function: the bladder, rectum, left and right femoral heads, and a ring of healthy tissue around the tumor (used to encourage conformity of the dose around the target). We choose $\beta=2$, $\bl_k=\bzero$ for all healthy organs, $\bl_{\mathcal{T}}=78\beee$, $\bu_k=81.9\beee$ for all healthy organs and the tumor, $\theta_k=30$ for the left and right femoral heads, and  $\theta_k=50$ for the other healthy organs, based on a protocol at Princess Margaret Cancer Centre in Toronto, Canada.

\subsection{Impact of different inverse optimization models on trade-off preservation}\label{sec:proof}
We used 24 historical treatments delivered at Princess Margaret Cancer Centre for our input $\hat\bx$ vectors.
Note that a clinical treatment plan is typically designed via the repeated solution of a nonconvex optimization problem with a larger number of objectives. Thus, the historical treatment plans are not likely to be Pareto optimal for problem~\eqref{eq:fop_IMRT}.  The goal of inverse optimization here is to find a weight vector such that problem~\eqref{eq:fop_IMRT} generates a treatment plan that is similar to the historical one, in terms of the achieved clinical metrics (i.e., objective values).

On average across the 24 patients,  $n=409$ and $m_k=8,157$, resulting in $6,313$ variables and $14,770$ constraints (after including auxiliary variables) of the forward problem~\eqref{eq:fop_IMRT}. For each of the treatments $\hat\bx$, we derived weights using KES($\hat\bx$), IOP$_r$($\hat\bx$), LIOP$_r$($\hat\bx$), and the SLP algorithm. For KES($\hat\bx$), we solved five different instances, corresponding to five different ways of normalizing the weight vector (i.e., $\alpha_k = 1, k = 1, \ldots, 5$). For the SLP algorithm, parameters and termination criteria were chosen based on~\cite{nonlin_alg2} (e.g., terminating when the $l_2$ norm difference between two consecutive iterates is less than 0.001) and $\hat\bx$ was used as an initial solution. All the models were solved using CPLEX 12.3 on a computer with a 3.07 GHz 12-core CPU and 32 GB of RAM.

\begin{table}[h]
\begin{center}
\caption{Comparing the model of~\citet{Keshavarz} with different weights fixed to one for patient \#1.}\label{tab:comparison_KES}
\resizebox{14cm}{!}{\begin{tabular}{rrrrrrrrrrrrrrr}
\hline
        &\multicolumn{2}{c}{KES1}&&\multicolumn{2}{c}{KES2}&&\multicolumn{2}{c}{KES3}&&\multicolumn{2}{c}{KES4}&&\multicolumn{2}{c}{KES5}\\\cmidrule{2-3}\cmidrule{5-6}\cmidrule{8-9}\cmidrule{11-12}\cmidrule{14-15}
OAR     &$\balpha^*$&	$\frac{f_k(\bx^*)}{f_k(\hat\bx)}$	&& $\balpha^*$&	 $\frac{f_k(\bx^*)}{f_k(\hat\bx)}$&&$\balpha^*$&	$\frac{f_k(\bx^*)}{f_k(\hat\bx)}$	 &&$\balpha^*$&	 $\frac{f_k(\bx^*)}{f_k(\hat\bx)}$	&&$\balpha^*$&	 $\frac{f_k(\bx^*)}{f_k(\hat\bx)}$	\\\hline
Blad	&	0.973	&	0.653	&&	0.003	&	0.837	&&	0.000	&	0.782	 &&	0.000	 &	 0.811	&&	 0.001	&	0.823	\\
Rect	&	0.004	&	1.000	&&	0.865	&	0.823	&&	0.000	&	0.949	&&	0.000	 &	0.963	 &&	 0.002	&	0.972	\\
LFem	&	0.000	&	1.143	&&	0.005	&	0.232	&&	0.992	&	0.000	&&	0.004	 &	0.001	 &&	 0.000	&	0.023	\\
RFem	&	0.000	&	0.671	&&	0.004	&	0.177	&&	0.004	&	0.002	&&	0.992	 &	0.000	 &&	 0.000	&	0.050	\\
Ring	&	0.023	&	0.929	&&	0.123	&	0.614	&&	0.004	&	0.412	&&	0.004	 &	0.411	 &&	 0.997	&	0.404	\\
\hline\multicolumn{15}{l}{{\footnotesize KES\# refers to KES with \#-th weight fixed to one.  Weights were renormalized so that they add up to one.}}
\end{tabular}}
\end{center}
\end{table}
Table~\ref{tab:comparison_KES} summarizes the results of applying the five variations of the KES($\hat\bx$) model to patient \#1. It can be seen that the imputed weight values are heavily dependent on which objective was used for normalization in KES($\hat\bx$). The weights in Table~\ref{tab:comparison_KES} suggest that the KES model gives the vast majority of the weight to the objective whose weight is used in the normalization constraint.  The results were similar for the remaining patients. For each instance, $f_k(\bx^*)/f_k(\hat\bx)$ denotes the ratio of the objective value achieved by $\balpha^*$ to the objective value associated with $\hat\bx$ for objective $k$.  In Table~\ref{tab:comparison_KES}, we see that there is no well-defined pattern in the ratios. The amount of improvement varies from objective to objective and from model to model. Furthermore, KES1 provides an example where the dose on some objectives stays the same or even increases while the dose to other organs decreases. Overall, the KES model finds a weight vector that puts an emphasis only on a single objective - whichever was used for normalization - and the resulting treatment plans seem to vary significantly.

\begin{table}[h]
\begin{center}
\caption{Comparison of the results from IOP$_r$, LIOP$_r$, and SLP algorithm for three patients.}\label{tab:comparison}
\resizebox{14cm}{!}{\begin{tabular}{lrrrrrrrrrrrrrrrrr}
\hline
	&	&\multicolumn{4}{c}{IOP$_r$} &&	\multicolumn{4}{c}{LIOP$_r$} &&	 \multicolumn{4}{c}{SLP} && FOP\\\cmidrule{3-6}\cmidrule{8-11}\cmidrule{13-16}\cmidrule{18-18}
\multirow{ 2}{*}{Pat}	&	\multirow{ 2}{*}{OAR}&	\multirow{ 2}{*}{$\epsilon^*$}	&	 \multirow{ 2}{*}{$\balpha^*$}	 &	 \multirow{ 2}{*}{$\frac{f_k(\bx^*)}{f_k(\hat\bx)}$}	&	 Time	 &&	 \multirow{ 2}{*}{$\epsilon^*$} & \multirow{ 2}{*}{$\balpha^*$} & \multirow{ 2}{*}{$\frac{f_k(\bx^*)}{f_k(\hat\bx)}$} & Time && \multirow{ 2}{*}{$\epsilon^*$} & \multirow{ 2}{*}{$\balpha^*$} & \multirow{ 2}{*}{$\frac{f_k(\bx^*)}{f_k(\hat\bx)}$}	&	 Time && Time\\
	&		&		&	 	     &			&(s)&&		&	    	     &            &(s)&&		 &	 		 &			&(s) & &(s)\\\hline
\multirow{5}{*}{1}
&	Blad	& \multirow{5}{*}{0.812}	&	0.014	&	0.812	&	\multirow{5}{*}{1582}	 && \multirow{5}{*}{0.769}	&	0.014	&	0.813	&	 \multirow{5}{*}{167}	 && \multirow{5}{*}{0.813}	 &	 0.014	&	0.815	&	 \multirow{5}{*}{\begin{tabular}{r}662\\$[$30$]$\end{tabular}} && \multirow{5}{*}{371}\\
	&	Rect	&	&	0.935	&	0.812	&	&&	&	0.936	&	0.812	&		 &&		 &	 0.935	&	0.812	&		 &&\\
	&	LFem&	&	0.002	&	0.812	&	&&	&	0.001	&	0.905	&		 &&		 &	 0.002	&	0.746	&		 &&\\
	&	RFem&	&	0.000	&	0.809	&	&&	&	0.001	&	0.695	&		 &&		 &	 0.000	&	0.932	&		 &&\\
	&	Ring	&	&	0.049	&	0.812	&	&&	&	0.048	&	0.816	&		 &&		 &	 0.049	&	0.813	&		 &&\\\hline
\multirow{5}{*}{2}
&	Blad	&\multirow{5}{*}{0.803}	&	0.932	&	0.803	&	\multirow{5}{*}{1384}	 &&\multirow{5}{*}{0.748}	&	0.937	&	0.802	&	 \multirow{5}{*}{82}	 &&\multirow{5}{*}{0.803}	 &	 0.933	&	0.803	&	 \multirow{5}{*}{\begin{tabular}{r}1000\\$[$56$]$\end{tabular}} &&\multirow{5}{*}{258}\\
	&	Rect	&		&	0.002	&	0.803	&		&&		&	0.008	&	0.754	&		 &&		 &	 0.002	&	0.802	&							 &&\\
	&	LFem	&		&	0.012	&	0.803	&		&&		&	0.012	&	0.771	&		 &&		 &	 0.011	&	0.846	&							 &&\\
	&	RFem	&		&	0.018	&	0.803	&		&&		&	0.015	&	0.965	&		 &&		 &	 0.018	&	0.808	&							 &&\\
	&	Ring	&		&	0.036	&	0.803	&		&&		&	0.028	&	0.825	&		 &&		 &	 0.036	&	0.803	&								 &&\\\hline
\multirow{5}{*}{3}
&	Blad	&	\multirow{5}{*}{0.712} &	0.471	&0.712	&\multirow{5}{*}{764}	 &&\multirow{5}{*}{0.617}	 &0.456	&0.713	&\multirow{5}{*}{33}	 &&\multirow{5}{*}{0.712}	&0.474	&0.712	 &\multirow{5}{*}{\begin{tabular}{r}435\\$[$30$]$\end{tabular}} &&\multirow{5}{*}{168}\\
	&	Rect	&		&0.412	&0.712	&		&&		&0.438	&0.709	&	&&		 &0.409	&0.712	& &&\\
	&	LFem&		&0.008	&0.712	&		&&		&0.012	&0.525	&	&&		 &0.007	&0.765	& &&\\
	&	RFem&		&0.033	&0.712	&		&&		&0.027	&0.889	&	&&		 &0.033	&0.689	& &&\\
	&	Ring	&		&0.076	&0.712	&		&&		&0.067	&0.746	&	&&		 &0.077	&0.709	& &&\\\hline	
\multicolumn{16}{l}{{\footnotesize [$\cdot$] denotes the number of iterations.}}\\
\end{tabular}}
\end{center}
\end{table}
Table~\ref{tab:comparison} shows the results from applying the IOP$_r$ model, LIOP$_r$ model, and SLP algorithm to three example patients.  For the SLP algorithm, the solution time is the total running time through all iterations.  As expected, the component-wise ratio obtained by the IOP$_r$ model is equivalent to $\epsilon^*$ unless the corresponding weight is zero in which case the ratio is at most $\epsilon^*$ (see Corollary~\ref{prop:relationship}).  The practical interpretation of this result is that although the ratio $f_k(\bx^*)/f_k(\hat\bx)=\epsilon^*$ may be achievable, the objective value can be further decreased (i.e., improved) without any sacrifice.  Note that preserving the trade-off using the IOP$_r$ model comes at a higher computational cost compared to the LIOP$_r$ model which approximately maintains the relative preservation.  Also note that $\epsilon^*$ from LIOP$_r$ is less than $\epsilon^*$ from IOP$_r$ as expected, because LIOP$_r$ is an outer approximation to IOP$_r$.  Finally, the SLP algorithm strikes a middle ground between the LIOP$_r$ and IOP$_r$ models in terms of trade-off preservation and computational efficiency.  

Figure~\ref{fig:scatterplot} shows the performance of KES, IOP$_r$, LIOP$_r$ and SLP across all 24 patients in terms of the solution time and trade-off preservation, the latter of which is quantified using the variance of the component-wise ratios $f_k(\bx^*)/f_k(\hat\bx)$ across $k$.  Variances less than $2^{-14} (<0.0001)$ were considered to represent perfect trade-off preservation so their values were rounded to zero.  For the KES results for each patient, we report the performance of the model with the normalization constraint $\alpha_{k'} = 1$, where $k'$ is the structure with the highest inverse weight determined by the IOP$_r$ model.  For example, in the case of patient \#1, the IOP$_r$ model shows that the rectum weight is the highest, so we use the inverse weights derived from the KES model that employs the normalization constraint $\alpha_2 = 1$ (the rectum is the second objective). In Figure~\ref{fig:scatterplot}, we can see that the KES model can generally be solved quickly but exhibits high variance in the objective function ratios, while the IOP$_r$ model generally has the longest solution time but preserves the initial relative trade-off vector well.  The three patients with variance in the range of $2^{-6}$ to $2^{-4}$ using the IOP$_r$ model had at least one zero-weighted objective. Finally, the LIOP$_r$ model and the SLP algorithm by extension strike a balance between the KES and IOP$_r$ models.  Table~\ref{tab:comparison_average} summarizes the numerical results across all patients, reinforcing the trade-offs between the models.
\begin{figure}[h]\centering
\includegraphics[width=120mm]{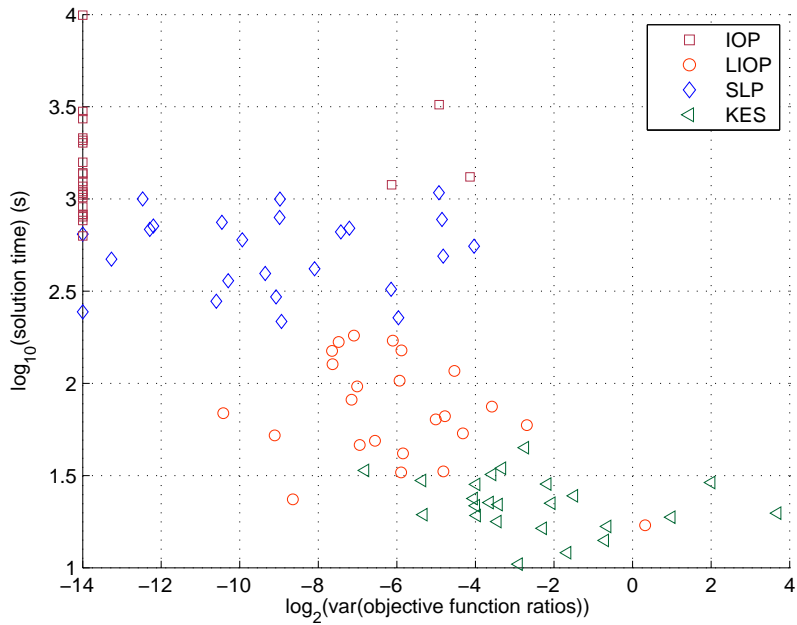}
\caption{Trade-off preservation and solution time between different inverse models}\label{fig:scatterplot}
\end{figure}
\begin{table}[h]
\begin{center}
\caption{Comparison of results averaged across 24 patients}\label{tab:comparison_average}
\begin{tabular}{crrrrrrr}
\hline
			&	IOP$_r$	&&	SLP		&&	LIOP$_r$	&&	KES	 \\\cmidrule{2-2}\cmidrule{4-4}\cmidrule{6-6}\cmidrule{8-8}
Var$\left(\frac{\bff(\bx^*)}{\bff(\hat\bx)}\right)$	&	0.004	&&	0.009	&&	0.076	&&	0.938	 \\
$|\epsilon^{{\text{{\tiny IOP}}}}-\epsilon^*|$		&	0	    &&	0.001	&&	0.050	&&	N/A	 \\
$||\balpha^{{\text{{\tiny IOP}}}}-\balpha^*||_2$ 	&	0	    &&	0.007	&&	0.077	&&	0.281	 \\
Solution time (s)			    					&	1,807	&&	570		&&	84		&&	23	\\\hline
\multicolumn{8}{l}{\footnotesize KES fixes the weight determined highest by the IOP$_r$ model to one.}\\
\multicolumn{8}{l}{\footnotesize  $\bx^{{\text{{\tiny IOP}}}}$, $\epsilon^{{\text{{\tiny IOP}}}}$ and $\balpha^{{\text{{\tiny IOP}}}}$ denote the optimal $\bx$, $\epsilon$ and $\balpha$ obtained by}\\
\multicolumn{8}{l}{\footnotesize  the IOP$_r$ model, respectively.}
\end{tabular}
\end{center}
\end{table}		

\subsection{Impact of different inverse optimization models on clinical metrics}\label{sec:clinical}
While Tables~\ref{tab:comparison_KES}, \ref{tab:comparison}, and \ref{tab:comparison_average} reinforce our theoretical results about the lack of trade-off preservation leading to inconsistent perturbations in objective space, we next comment on how these results translate into the clinical context for radiation therapy. In particular, we examine the impact of trade-off preservation on the resulting treatment plan's ability to replicate the historical treatment plan with respect to key clinical metrics.  

We first compared treatment plans generated by the KES weights to the clinical treatment plan for patient \#1.  We use a dose volume histogram (DVH), which shows for each plan the fractional volume of an organ that receives a certain dose or higher, to compare the KES and clinical plans.  Figure~\ref{fig:DVH_KES1_clinical} shows that while the solution using the KES1 weights (dashed lines) has significantly improved the bladder dose (dashed green), the rectum dose is generally worse, violating the clinical acceptability requirement of Princess Margaret Cancer Centre that less than 50\% of the volume receive more than 50 Gy. On the other hand, Figure~\ref{fig:DVH_IOP_clinical} shows that the treatment plan generated by the IOP$_r$ weights for patient \#1 is much more similar to the clinical plan over all the four organs at risk. Given that the bladder, rectum, and the left/right femoral head objectives were to minimize radiation dose exceeding 50 Gy, 50 Gy, and 30 Gy, respectively, the DVHs of the IOP$_r$ plan show a similar dose reduction across the organs compared to the clinical DVHs above the dose thresholds. The results were similar for the remaining patients.
\begin{figure}[h]\centering
\includegraphics[width=130mm]{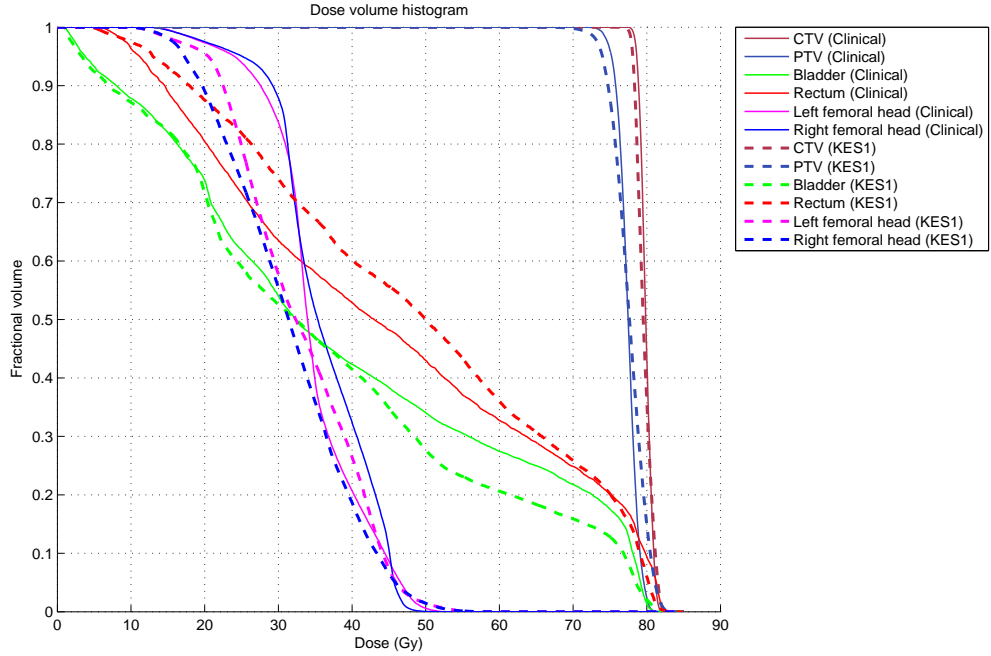}
\caption{DVHs of the clinical plan and KES1 weights for patient \#1}\label{fig:DVH_KES1_clinical}
\end{figure}
\begin{figure}[h!]\centering
\includegraphics[width=130mm]{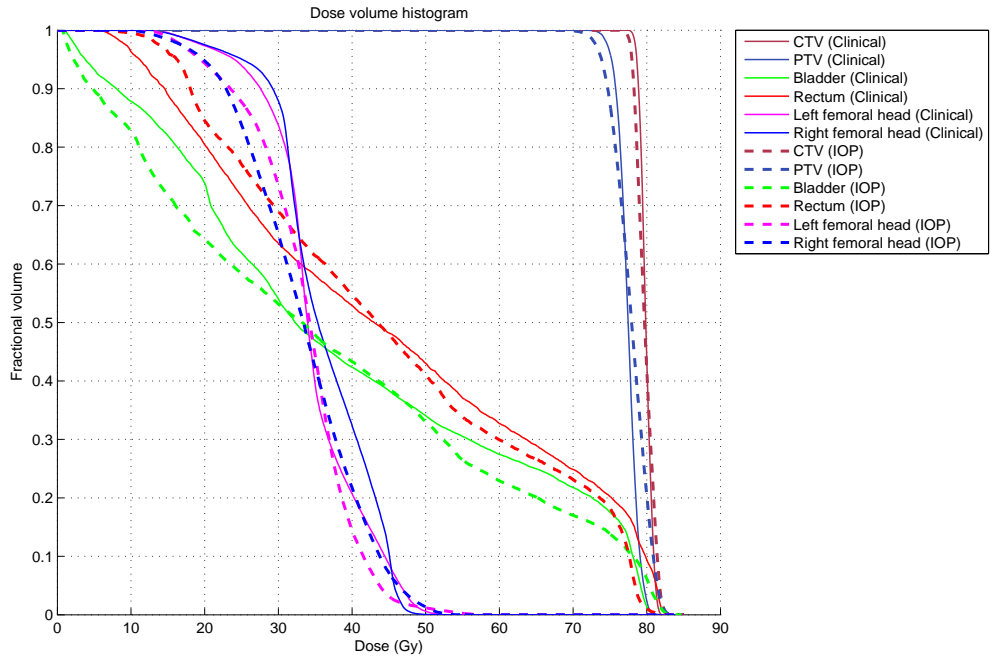}
\caption{DVHs of the clinical plan and IOP$_r$ weights for patient \#1}\label{fig:DVH_IOP_clinical}
\end{figure}
		
The ability to accurately replicate a clinical plan, including the relevant clinical trade-offs encoded in the plan, is an important feature of any knowledge-based, automated planning system. Such approaches are becoming increasingly popular as a means to efficiently create treatments for new patients. For example, given a new patient, the system can identify a ``similar'' historical patient (e.g., anatomically) and produce an appropriate weight vector using our trade-off preserving inverse optimization model. Alternatively, one can use a prediction model, trained using inversely optimized weight vectors and anatomy from historical patients, to generate an appropriate weight vector for a new patient~\citep{geom2,geom3,geom4,lee_geom,bout_geom}. Such an approach is more likely to produce plans with acceptable clinical performance, assuming the historical plans were acceptable.

\section{Conclusion}
Trade-offs between objectives are critical in multi-objective optimization.  In this paper, we developed a new approach to inverse convex multi-objective optimization that explicitly considers preserving trade-off encoded in a given initial solution.  Our approach is general and maintains the complexity of the forward problem (i.e., is convex). The notion of trade-off preservation is generally applicable to multi-objective optimization problems where a decision maker's intentions are encoded in the objective values of a given solution, and provides a connection with a relative or absolute optimality gap with respect to the given solution.  We formulated a general inverse convex optimization framework that encompasses many of the inverse models in the literature and demonstrated how an inverse model can be adjusted to preserve the initial trade-off.  We compared various inverse optimization models computationally using data from prostate cancer radiation therapy and demonstrated the ability of our inverse model to preserve the clinical quality of a historical treatment using imputed weights, which is important in the design of next generation automated treatment planning systems.

\section*{Acknowledgements} The authors would like to thank Dr. Michael B. Sharpe and Dr. Tim Craig at Princess Margaret Cancer Centre for providing the clinical data and expertise.

\bibliography{InvConvex}
\bibliographystyle{chicago}

\end{document}